\documentclass[11pt]{amsart}
\usepackage{amssymb}
\usepackage{dsfont}
\usepackage{color}

\newtheorem{theorem}{Theorem}
\newtheorem{lemma}[theorem]{Lemma}

\newtheorem{problem}[theorem]{Problem}
\newtheorem{corollary}[theorem]{Corollary}

\newtheorem{remark}[theorem]{Remark}

\newcommand{\N}{\mathbb N}
\newcommand{\Q}{\mathbb Q}

\newcommand{\R}{\mathbb R}

\newcommand{\C}{\mathbb C}

\newcommand{\on}{\operatorname}
\author{Taras Banakh}

\address{T.Banakh: Ivan Franko University of Lviv (Ukraine) and Jan Kochanowski Uniwersity in Kielce (Poland)}
\email{t.o.banakh@gmail.com}

\author{Artur Bartoszewicz}

\address{Institute of Mathematics, Technical University of \L\'od\'z, W\'olcza\'nska 215, 93-005
\L\'od\'z, Poland}
\email {arturbar@p.lodz.pl}

\author{Szymon G\l \c ab}

\address{Institute of Mathematics, Technical University of \L\'od\'z, W\'olcza\'nska 215, 93-005
\L\'od\'z, Poland}
\email {szymon.glab@p.lodz.pl}

\author{Emilia Szymonik}

\address{Institute of Mathematics, Technical University of \L\'od\'z, W\'olcza\'nska 215, 93-005
\L\'od\'z, Poland}
\email {szymonikemilka@wp.pl}

\title{Algebraic and topological properties of some sets in $\ell_1$}
\date{}

\begin{document}

\begin{abstract}
For a sequence $x \in\ell_1 \setminus c_{00}$, one can consider the set $E(x)$ of all subsums of series $\sum_{n=1}^{\infty} x(n)$. Guthrie and Nymann proved that $E(x)$ is one of the following types of sets:
\begin{itemize}
\item[($\mathcal I$)] a finite union of closed intervals;
\item[($\mathcal C$)] homeomorphic to the Cantor set;
\item[($\mathcal{MC}$)] homeomorphic to the set $T$ of subsums of
$\sum_{n=1}^\infty b(n)$ where $b(2n-1) = 3/4^n$ and $b(2n) = 2/4^n$.
\end{itemize}
By $\mathcal I$, $\mathcal C$ and $\mathcal{MC}$ denote the sets of all sequences $x \in\ell_1 \setminus c_{00}$, such that $E(x)$ has the property ($\mathcal I$), ($\mathcal C$) and ($\mathcal{MC}$), respectively. In this note we show that $\mathcal I$ and $\mathcal C$ are strongly $\mathfrak{c}$-algebrable and $\mathcal{MC}$ is $\mathfrak{c}$-lineable.
We show that $\mathcal C$ is a dense $G_\delta$-set in $\ell_1$ and $\mathcal I$ is a true $\mathcal F_\sigma$-set. Finally  we show that $\mathcal I$ is spaceable while $\mathcal C$ is not spaceable.
\end{abstract}
\maketitle

\section{Introduction}

\subsection{Lineability, algebrability and spaceability}
Having a linear algebra $A$ and its subset $E\subset A$ one can ask if $E\cup\{0\}$ contains a linear subalgebra $A'$ of $A$. Roughly speaking if the answer is positive, then $E$ is algebrable. It is a recent trend in Mathematical Analysis to establish the algebrability of sets $E$ which are far from being linear, that is $x,y\in E$ does not generally imply $x+y\in E$. Such algebrability results were obtained in sequence spaces (see \cite{BG}, \cite{BGP} and \cite{BG1}) and in function spaces (see \cite{ACPSS}, \cite{ASS}, \cite{APSS}, \cite{GPMSS} and \cite{GPPSS}).

Assume that $V$ is a linear space (linear algebra). A subset $E \subset V $ is called lineable (algebrable) whenever $E \cup \{0\}$ contains an infinite-dimensional linear space (infinitely generated linear algebra, respectively), see \cite{AGS}, \cite{LBG} and \cite{GQ}. For a cardinal $\kappa > \omega$, let us observe that the set $E$ is $\kappa$-algebrable (i.e. it contains $\kappa$-generated linear algebra), if and only if it contains an algebra which is a $\kappa$-dimensional linear space (see \cite{BG}). Moreover, we say that a subset $E$ of a commutative linear algebra $V$ is strongly $\kappa$-algebrable (\cite{BG}), if there exists a $\kappa$-generated free algebra $A$ contained in $E \cup \{0\}$.

Note, that $X= \{x_\alpha : \alpha < \kappa \} \subset E$ is a set of free generators of a free algebra $A \subset E$ if and only if the set $X'$ of elements of the form $x_{\alpha_1}^{k_1} x_{\alpha_2}^{k_2}\dots x_{\alpha_n}^{k_n}$ is linearly independent and all linear combinations of elements from $X'$ are in $E \cup \{0 \}$. It is easy to see that free algebras have no divisors of zero.

In practice, to prove $\kappa$-algebrability of set $E \subset V$ we have to find $X \subseteq E$ of cardinality $\kappa$ such that for any polynomial $P$ in $n$ variables and any distinct $x_1,\dots,x_n \in X$ we have either $P(x_1,\dots,x_n) \in E$ or $P(x_1,\dots,x_n)=0$. To prove the strong $\kappa$-algebrability of $E$ we have to find $X \subset E$, $|X|= \kappa$, such that for any non-zero polynomial P and distinct $x_1,\dots,x_n \in X$ we have $P(x_1,\dots,x_n) \in E$.

In general, there are subsets of linear algebras which are algebrable but not strongly algebrable. Let $c_{00}$ be a subset of $c_0$ consisting of all sequences with real terms equal to zero from some place. Then the set $c_{00}$ is algebrable in $c_0$ but is not strongly $1$-algebrable \cite{BG}.

Let $X$ be a Banach space. The subset $M$ of $X$ is spaceable if $M\cup\{0\}$ contains infinitely dimensional closed subspace $Y$ of $X$. Since every infinitely dimensional Banach space contains linearly independent set of the cardinality continuum, the spaceability implies $\mathfrak{c}$-lineability. However, the spaceability is a much stronger property then $\mathfrak{c}$-lineability. The notions of spaceability and $\mathfrak{c}$-algebrability are incomparable.
We will show that even $\mathfrak{c}$-algebrable dense $\mathcal G_\delta$-sets in $\ell_1$ may not be spaceable. On the other hand, there are sets in $c_0$ which are spaceable but not $1$-algebrable (see  \cite{BG}).

\subsection{The subsums of series}

Let $x\in \ell_1$. The set of all subsums of $\sum_{n=1}^\infty x(n)$, meaning the set of sums of all subseries of $\sum_{n=1}^\infty x(n)$, is defined by
$$
E(x)=\{a\in{\R}:\exists A\subset\N\;\;\; \sum_{n\in A}x(n)=a\}.
$$
Some authors call it the achievement set of $x$.
The following theorem is due to Kakeya.

\begin{theorem} \label{kakeya}\cite{K}.
Let $x \in \ell_1$
\item[(1)] If $x\not\in c_{00}$, then $E(x)$ is a perfect compact set.
\item[(2)] If $|x(n)|>\sum_{i>n}|x(i)|$ for almost all $n$, then $E(x)$ is homeomorphic to the Cantor set.
\item[(3)] If $|x(n)|\leq \sum_{i>n}|x(i)|$ for $n$ sufficiently large, then $E(x)$ is a finite union of closed intervals. In the case of non-increasing sequence $x$, the last inequality is also necessary to obtain $E(x)$ being a finite union of intervals.
\end{theorem}

Moreover, Kakeya conjectured that $E(x)$ is either nowhere dense or it is a finite union of intervals. Probably, the first counterexample to this conjecture was given (without a proof) by Weinstein and Shapiro \cite{WS} and, with a correct proof, by Ferens \cite{F}.
Guthrie and Nymann \cite{GN} showed that, for the sequence $b$ given by the formulas $b(2n-1)= \frac{3}{4^n}$ and $b(2n)= \frac{2}{4^n}$, the set $T=E(b)$ is not a finite union of intervals but it has nonempty interior. In the same paper they formulated the following theorem
\begin{theorem}\label{1} \cite{GN}
Let $x\in \ell_1 \setminus c_{00}$, then $E(x)$ is one of the following sets:
\begin{itemize}
\item[(i)] a finite union of closed intervals;
\item[(ii)] homeomorphic to the Cantor set;
\item[(iii)] homeomorphic to the set $T$.
\end{itemize}
\end{theorem}

A correct proof of the Guthrie and Nymann trichotomy was given by Nymann and S\'aenz \cite{NS}. The sets homeomorphic to $T$ are called Cantorvals (more precisely: M-Cantorvals). Note that Theorem \ref{1} can be formulated as follows:
The space $\ell_1$ is a disjoint union of the sets $c_{00},\mathcal I,\mathcal C$ and $\mathcal{MC}$ where $\mathcal I$ consists of sequences $x$ with $E(x)$ equal to a finite union of intervals, $\mathcal C$ consists of sequences $x$ with $E(x)$ homeomorphic to the Cantor set, and $\mathcal{MC}$ of $x$ with $E(x)$ being an M-Cantorval.
\\For $x \in \ell_1$, let $x'$ be an arbitrary finite modification of $x$, and let $|x|$ denote the sequence $y \in \ell_1$ such that $y(n)=|x(n)|$. Then $x \in \mathcal I \iff |x| \in \mathcal{I} \iff x' \in \mathcal{I}$. The same equivalences hold for sets $\mathcal C$ and $\mathcal{MC}$.

\section{Algebraic substructures in $\mathcal C$, $\mathcal I$ and $\mathcal{MC}$.}

Jones in a very nice paper \cite{J} gives the following example. Let $x(n)=1/2^n$ and $y(n)=1/3^n$. Then clearly $x\in\mathcal I$ and $y\in\mathcal C$. Moreover, $x+y\in\mathcal C$ and $x-y\in\mathcal I$. Since $x=(x+y)-y$ and $y=-(x-y)+x$, then neither $\mathcal I$ nor $\mathcal C$ is closed under pointwise addition.
However, in the present paper we show that the sets $\mathcal C$, $\mathcal I$ and $\mathcal{MC}$ contain large ($\mathfrak{c}$-generated) algebraic structures. To prove the strong $\mathfrak{c}$-algebrability of $\mathcal C$ and $\mathcal I$, we will combine Theorem \ref{kakeya} and the method of linearly independent exponents, which was successful in \cite{BGP} and \cite{BG}. In the next theorem we construct generators as the powers of one geometric series $x_q$ ($x_q(n)=q^n$) for $0<q< \frac{1}{2}$. Clearly, by Theorem \ref{kakeya}, $x_q \in \mathcal C$.

\begin{theorem}
$\mathcal C$ is strongly $\mathfrak{c}$-algebrable.
\end{theorem}

\begin{proof}

Fix $q \in (0,1/2)$. Let $\{r_\alpha:\alpha<\mathfrak{c}\}$ be a linearly independent (over the field of all rationals $\Q$) set of reals greater than $1$. Let $x_\alpha(n)=q^{r_\alpha n}$. We will show that the set $\{x_\alpha:\alpha<\mathfrak{c}\}$ generates a free algebra $\mathcal A$ which, except for the null sequence, is contained in $\mathcal{C}$.

To do this, we will show that for any $\beta_1,\beta_2,\dots,\beta_m\in\R \setminus \{0 \}$, any matrix $[k_{il}]_{i\leq m,l\leq j}$ of natural numbers with nonzero distinct rows, and any $\alpha_1<\alpha_2<\dots<\alpha_j<\mathfrak{c}$, the sequence $x$ given by
$$x(n)=P(x_{\alpha_1},\dots,x_{\alpha_j})(n)$$ where
$$P(z_1, \dots ,z_j)=\beta_1 z_1^{k_{11}} z_2^{k_{12}} \dots z_j^{k_{1j}}+ \dots +\beta_m z_1^{k_{m1}} z_2^{k_{m2}} \dots z_j^{k_{mj}}$$
is in $\mathcal C$. In other words,
$$
x(n)=\beta_1q^{n(r_{\alpha_1}k_{11}+ \dots +r_{\alpha_j}k_{1j})}+ \dots +
\beta_mq^{n(r_{\alpha_1}k_{m1}+ \dots +r_{\alpha_j}k_{mj})}
$$

Since $r_{\alpha_1}, \dots ,r_{\alpha_j}$ are linearly independent and the rows of $[k_{il}]_{i \leqslant m, l \leqslant j}$ are distinct, the numbers $r_1:=r_{\alpha_1}k_{11}+ \dots +r_{\alpha_j}k_{1j}, \dots ,r_m:=r_{\alpha_1}k_{m1}+ \dots +r_{\alpha_j}k_{mj}$ are distinct. We may assume that $r_1< \dots <r_m$. Then
$$
\frac{|x(n)|}{\sum_{i>n}|x(i)|}=\frac{|\beta_1q^{nr_1}+ \dots +
\beta_mq^{nr_m}|}{\sum_{i>n}|\beta_1q^{ir_1}+ \dots +
\beta_mq^{ir_m}|}$$
$$\geq\frac{|\beta_1q^{nr_1}+ \dots +
\beta_mq^{nr_m}|}{\sum_{i>n}(|\beta_1|q^{ir_1}+ \dots +
|\beta_m|q^{ir_m})} =\frac{|\beta_1q^{nr_1}+ \dots +
\beta_mq^{nr_m}|}{\frac{|\beta_1|q^{(n+1)r_1}}{1-q^{r_1}}+ \dots +
\frac{|\beta_m|q^{(n+1)r_m}}{1-q^{r_m}}}$$
$$\to\frac{1-q^{r_1}}{q^{r_1}}>1.
$$
Therefore there is $n_0$, such that $|x(n)|>\sum_{i>n}|x(i)|$ for all $n\geq n_0$. Hence, by Theorem \ref{kakeya}, we obtain that $x\in\mathcal{C}$.

\end{proof}

It is obvious that the geometric sequence $x_q$, even for $q> \frac{1}{2}$, is not useful to construct the generators of linear algebra contained in $\mathcal{I}$. Indeed, for sufficiently large exponent $k$, the sequence $x_q^k$ belongs to $\mathcal{C}$. So, in the next theorem we use the harmonic series.

\begin{theorem}

$\mathcal I$ is strongly $\mathfrak{c}$-algebrable.

\end{theorem}

\begin{proof}

Let $K$ be a linearly independent subset of $(1,\infty)$ of cardinality $\mathfrak{c}$. For $\alpha \in K$, let $x_{\alpha}$ be a sequence given by the formula   $x_{\alpha}(n)=\frac{1}{n^{\alpha}}$.
 We will show that the set $\{x_\alpha:\alpha \in K\}$ generates a free algebra $\mathcal A$ which is contained in $\mathcal{I} \cup \{0\}$.
To do this, we will show that for any $\beta_1,\beta_2, \dots ,\beta_m\in\R \setminus \{0\}$, any matrix $[k_{il}]_{i\leq m,l\leq j}$ of natural numbers with nonzero distinct rows, and any $\alpha_1<\alpha_2< \dots <\alpha_j$, the sequence $x$ defined by
$$
x=P(x_{\alpha_1},x_{\alpha_2}, \dots ,x_{\alpha_j})
$$
$$
=\beta_1 x_{\alpha_1}^{k_{11}} x_{\alpha_2}^{k_{12}} \dots  x_{\alpha_j}^{k_{1j}} +\beta_2 x_{\alpha_1}^{k_{21}} x_{\alpha_2}^{k_{22}} \dots  x_{\alpha_j}^{k_{2j}} + \dots +\beta_m x_{\alpha_1}^{k_{m1}} x_{\alpha_2}^{k_{m2}} \dots x_{\alpha_j}^{k_{mj}}
$$
belongs to $\mathcal{I}$. We have
$$
x(n)=P(x_{\alpha_1}, x_{\alpha_2}, \dots , x_{\alpha_j})(n)
$$
$$
=\beta_1 \frac{1}{n^{\alpha_1 k_{11}+ \alpha_2 k_{12}+ \dots + \alpha_j k_{1j}}} + \dots +\beta_m \frac{1}{n^{\alpha_1 k_{m1}+ \alpha_2 k_{m2}+ \dots +\alpha_j k_{mj}}}
$$
$$
=\beta_1 \frac{1}{n^{p_1}} + \beta_2 \frac{1}{n^{p_2}} + \dots +\beta_j \frac{1}{n^{p_m}}
$$
Note that $p_1, \dots ,p_m$ are distinct. Assume that $p_1 < p_2 < \dots < p_m$. We have
$$
\frac{|x(n)|}{\sum_{k>n}|x(k)|} = \frac{|\beta_1 \frac{1}{n^{p_1}} +\beta_2 \frac{1}{n^{p_2}} + \dots + \beta_m \frac{1}{n^{p_m}}|}{\sum_{k>n}|\beta_1 \frac{1}{k^{p_1}} +\beta_2 \frac{1}{k^{p_2}} + \dots +\beta_m \frac{1}{k^{p_m}}|}
$$
$$
\leqslant \frac{|\beta_1 \frac{1}{n^{p_1}} +\beta_2 \frac{1}{n^{p_2}} + \dots +\beta_m \frac{1}{n^{p_m}}|}{\sum_{k>n}\left(|\beta_1 \frac{1}{k^{p_1}}| -|\beta_2 \frac{1}{k^{p_2}}| - \dots -|\beta_m \frac{1}{k^{p_m}}|\right)}
$$
$$ \leqslant \frac{|\beta_1 \frac{1}{n^{p_1}} +\beta_2 \frac{1}{n^{p_2}} + \dots +\beta_m \frac{1}{n^{p_m}}|}{|\beta_1|\int_{n+1}^{\infty} \frac{1}{x^{p_1}} dx -|\beta_2|\int_{n}^{\infty} \frac{1}{x^{p_2}} dx - \dots -|\beta_m| \int_{n}^{\infty} \frac{1}{x^{p_m}} dx} $$
$$=\frac{|\beta_1 +\beta_2 \frac{n^{p_1}}{n^{p_2}} + \dots +\beta_m \frac{n^{p_1}}{n^{p_m}}|}{n[|\beta_1| \frac{1}{p_1-1} \frac{n^{p_1-1}}{({n+1})^{{p_1}-1}} -|\beta_2| \frac{1}{p_2-1} \frac{n^{p_1-1}}{({n})^{{p_2}-1}}  - \dots -|\beta_m| \frac{1}{p_m-1} \frac{n^{p_1-1}}{({n})^{{p_m}-1}}]}$$ $$\mathop{\longrightarrow}_{n \to \infty} 0 <1.
$$
Observe that the first inequality holds for $n$ large enough.
Therefore there is $n_0$ such that $|x(n)|\leq\sum_{i>n}|x(i)|$ for any $n\geq n_0$. Hence, by Theorem \ref{kakeya} we obtain that $x\in\mathcal{I}$.

\end{proof}

The method described in the next lemma belongs to the mathematical folklore and was used to construct sequences $x$'s with $E(x)$ being Cantorvals. We present its proof since we did not find it explicitly formulated in the mathematical literature.

\begin{lemma}\label{lemmaMC}
Let $x \in \ell_1$ be such that
\begin{itemize}
\item[(i)] $E(x)$ contains an interval;
\item[(ii)] $|x(n)|> \sum_{i>n}|x(i)|$ for infinitely many $n$;
\item[(iii)] $|x_n| \geqslant |x_{n+1}|$ for almost all $n$.
\end{itemize}
Then $x \in \mathcal{MC}$.
\end{lemma}

\begin{proof}
By (ii)-(iii), the point $x$ does not belong to $\mathcal{I}$. By (i), the point $x$ does not belong to $\mathcal{C}$. Hence, by Theorem \ref{1} we get $x \in \mathcal{MC}$.
\end{proof}

Up to last years, there were only known a few examples of sequences belonging to $\mathcal{MC}$. These examples were not very useful to construct a large number of linearly independent sequences. Recently, Jones in \cite{J} has constructed a one-parameter family of sequences in $\mathcal{MC}$. We shall use some modification of the example given by Jones in the proof of our next theorem.

\begin{theorem}

$\mathcal{MC}$ is $\mathfrak{c}$-lineable.

\end{theorem}

\begin{proof}
Let
$$
x_q = (4,3,2,4q,3q,2q,4q^2,3q^2,2q^2,4q^3, \dots)
$$
\\ and
$$
y_q = (1,1,1,1,1,q,q,q,q,q,q^2,q^2,q^2,q^2,q^2,q^3, \dots)
$$ for $q\in[\frac{1}{6},\frac{2}{11}) $.
\\Observe that the sequences $x_q, q\in[\frac{1}{6},\frac{2}{11})$ are linearly independent. We need to show that each non-zero linear combination of sequences $x_q$ fulfils the assumptions (i)--(iii) of Lemma \ref{lemmaMC} and therefore it is actually in $\mathcal{MC}$. To prove this, let us fix $q_1 > q_2 > \dots > q_m\in[\frac{1}{6},\frac{2}{11})$, $\beta_1,\beta_2, \dots ,\beta_m\in\R$ and define sequences $x$ and $y$ by
$$
x(n)=\beta_1x_{q_1}(n)+\beta_2x_{q_2}(n)+ \dots +\beta_mx_{q_m}(n)
$$
and
$$
y(n)=\beta_1y_{q_1}(n)+\beta_2y_{q_2}(n)+ \dots +\beta_my_{q_m}(n).
$$

At first, we will check that for almost all $n$
\begin{equation} \label{eq1}
2|\beta_1{q_1}^n+\beta_2{q_2}^n+ \dots +\beta_m{q_m}^n|>9\sum_{k>n}|\beta_1{q_1}^k+\beta_2{q_2}^k+ \dots +\beta_m{q_m}^k|.
\end{equation}
We have
$$
\frac{2|\beta_1{q_1}^n+\beta_2{q_2}^n+ \dots +\beta_m{q_m}^n|}{9\sum_{k>n}|\beta_1{q_1}^k+\beta_2{q_2}^k+ \dots +\beta_m{q_m}^k|} \geqslant \frac{2|\beta_1{q_1}^n+\beta_2{q_2}^n+ \dots +\beta_m{q_m}^n|}{9\sum_{k>n}|\beta_1{q_1}^k|+|\beta_2{q_2}^k|+ \dots +|\beta_m{q_m}^k|}$$
$$=\frac{2|\beta_1{q_1}^n+\beta_2{q_2}^n+ \dots +\beta_m{q_m}^n|}{9(|\beta_1|\frac{{q_1}^{n+1}}{1-q_1}+|\beta_2|\frac{{q_2}^{n+1}}{1-q_2}+ \dots +|\beta_m|\frac{{q_m}^{n+1}}{1-q_m})} \mathop{\longrightarrow}_{n \to \infty} \frac{2}{9}\cdot\frac{1-q_1}{q_1} > \frac{2}{9}\cdot\frac{1-\frac{2}{11}}{\frac{2}{11}}= 1.
$$
Note that if $n$ is not divisible by $3$, then $|x(n)| \geqslant |x(n+1)|$. On the other hand, if $n=3l$, then
$$|x(n)|=2|\beta_1q_1^l+ \dots +\beta_mq_m^l|$$ and $$|x(n+1)|=3|\beta_1q_1^{l+1}+ \dots +\beta_mq_m^{l+1})| \leqslant 9\sum_{k>l}|\beta_1q_1^k+ \dots +\beta_mq_m^k|.$$
Hence by \eqref{eq1} we obtain $|x(n)| \geqslant |x(n+1)|$ for almost all $n$. By \eqref{eq1} we also have $|x(n)| > \sum_{i>n} |x(i)| $ for infinitely many $n$.
\\Now we will show that
\begin{equation} \label{eq2}
|\beta_1{q_1}^n+\beta_2{q_2}^n+ \dots +\beta_m{q_m}^n| \leqslant 5\sum_{k>n}|\beta_1{q_1}^k+\beta_2{q_2}^k+ \dots +\beta_m{q_m}^k|.
\end{equation}
We have
$$
\frac{|\beta_1{q_1}^n+\beta_2{q_2}^n+ \dots +\beta_m{q_m}^n|}{5\sum_{k>n}|\beta_1{q_1}^k+\beta_2{q_2}^k+ \dots +\beta_m{q_m}^k|} \leqslant \frac{|\beta_1{q_1}^n+\beta_2{q_2}^n+ \dots +\beta_m{q_m}^n|}{5|\sum_{k>n}\beta_1{q_1}^k+\beta_2{q_2}^k+ \dots +\beta_m{q_m}^k|}$$
$$=\frac{|\beta_1+\beta_2 (\frac{q_2}{q_1})^n + \dots +\beta_m(\frac{q_m}{q_1})^n|}{5|\beta_1\sum_{i>0}{q_1}^i+\beta_2 (\frac{q_2}{q_1})^n \sum_{i>0} {q_2}^i+ \dots +\beta_m (\frac{q_m}{q_1})^n \sum_{i>0} {q_m}^i|} $$
$$ \mathop{\longrightarrow}_{n \to \infty} \frac{1}{5}\cdot \frac{1-q_1}{q_1} \leqslant \frac{1}{5}\cdot\frac{1-\frac{1}{6}}{\frac{1}{6}}= 1.
$$
By \eqref{eq2} we obtain that $|y(n)| \leqslant \sum_{k>n} |y(k)| $ for almost all $n$. Therefore by Theorem \ref{kakeya}, the set $E(y)$ is a finite union of closed intervals. Thus $E(y)$ has non-empty interior.

To end the proof we need to show that $E(x)$ has non-empty interior. We will prove that
$$ 2 \sum_{n=0} (\beta_1{q_1}^n+\beta_2{q_2}^n+ \dots +\beta_m{q_m}^n) + E(y) \subseteq E(x).
$$Let
$$ t \in 2 \sum_{n=0} (\beta_1{q_1}^n+\beta_2{q_2}^n+ \dots +\beta_m{q_m}^n) + E(y). $$ Note that any element $t$ of $E(y)$ is of the form
$$ t=k_0(\beta_1+\beta_2+ \dots +\beta_m) + k_1 (\beta_1{q_1}+\beta_2{q_2}+ \dots +\beta_m{q_m}) $$
$$+ k_2 (\beta_1{q_1}^2+\beta_2{q_2}^2+ \dots +\beta_m{q_m}^2)+ \dots$$ where $k_n \in \{0,1,2,3,4,5\}$. Thus $t$ is of the form
$$t= 2 \sum_{n=0} (\beta_1{q_1}^n+\beta_2{q_2}^n+ \dots +\beta_m{q_m}^n)+$$
$$+ [ k_0(\beta_1+\beta_2+ \dots +\beta_m) + k_1 (\beta_1{q_1}+\beta_2{q_2}+ \dots +\beta_m{q_m})$$
$$ + k_2 (\beta_1{q_1}^2+\beta_2{q_2}^2+ \dots +\beta_m{q_m}^2)+\dots  ]$$
$$=(2+k_0)(\beta_1+\beta_2+ \dots +\beta_m)+(2+k_1)(\beta_1{q_1}+\beta_2{q_2}+ \dots+\beta_m{q_m})+$$ $$+(2+k_2)(\beta_1{q_1}^2+\beta_2{q_2}^2+ \dots +\beta_m{q_m}^2)+ \dots
$$
Note that each number from $\{2,3,4,5,6,7\}$, that is every number of the form $2+k_n$, can be written as a sum of numbers $4,3,2$. Hence $t \in E(x)$ and $E(x)$ has non-empty interior. So $x \in \mathcal{MC}$.
\end{proof}

\section{The topological size and Borel classification of  $\mathcal C$, $\mathcal I$ and $\mathcal{MC}$.}

Let us observe that all the sets $c_{00}$,  $\mathcal C$, $\mathcal I$ and $\mathcal{MC}$ are dense in $\ell_1$. Moreover, $c_{00}$ is an $\mathcal F_{\sigma}$-set of the first category. We are interested in studying the topological size and Borel classification of considered sets. To do it, let us consider the hyperspace $H(\R)$, that is the space of all non-empty compact subsets of reals, equipped with the Vietoris topology (see \cite{Ke}, 4F, pp.24-28). Recall, that the Vietoris topology is generated by the subbase of sets of the  form $\{ K\in H(\R): K \subset U \}$ and $\{ K\in H(\R): K \cap U \neq \emptyset \}$ for all open sets $U$ in $\R$.
This topology is metrizable by the Hausdorff metric $d_H$ given by the formula
$$
d_H(A,B)= \max \{ \max_{t \in A} d(t,B), \max_{s \in B} d(s,A) \}
$$
where $d$ is the natural metric in $\R$.
It is known that the set $N$ of all nowhere dense compact sets is a $G_\delta$-set in $H(\R)$ and the set $F$ of all compact sets with finite number of connected components is an $\mathcal F_{\sigma}$-set.
To see this, it is enough to observe  that
\begin{itemize}
\item $K$ is nowhere dense if and only if for any set $U_n$ from a fixed countable base of natural topology in $\R$ there exists a set $U_m$ from this base, such that $cl(U_m) \subset U_n$ and $K \subset (cl(U_m))^c$;
\item $K$ has more then $k$ components if and only if there exist pairwise disjoint open intervals $J_1,J_2, \dots ,J_{k+1}$, such that $K \subset J_1 \cup J_2 \cup \dots \cup J_{k+1}$ and $K \cap J_i \neq \emptyset $ for $i=1,2, \dots ,k+1$.
\end{itemize}
Now, let us observe that if we assign the set $E(x)$ to the sequence $x \in \ell_1$, we actually define the function $E: \ell_1 \to H(\R)$.

\begin{lemma}
The function $E$ is Lipschitz with Lipschitz constant $L=1$, hence it is continuous.
\end{lemma}

\begin{proof}
Let $t \in E(x)$. Then there exists a subset $A$ of $\N$ such that $t= \sum_{n \in A} x(n).$ We have
$$d(t,E(y)) \leqslant d(t,\sum_{n \in A} y(n))=\left|\sum_{n \in A} (x(n)-y(n))\right| \leqslant \sum_{n \in \N} |(x(n)-y(n))|= \Vert x-y \Vert _1$$
where $\Vert\cdot\Vert_1$ denotes the norm in $\ell_1$. Hence,
$ d_H(E(x),E(y)) \leqslant \Vert x-y \Vert_1 $.
\end{proof}

\begin{theorem}
The set $\mathcal C$ is a dense $G_\delta$-set (and hence residual), $\mathcal I$ is a true $\mathcal F_\sigma$-set (i.e. it is $\mathcal F_\sigma$ but not $\mathcal G_\delta$) of the first category, and $\mathcal{MC}$ is in the class $(\mathcal{F}_{\sigma\delta}\cap\mathcal{G}_{\delta\sigma})\setminus\mathcal G_\delta$.
\end{theorem}

\begin{proof}
Let us observe that $ \mathcal{C} \cup c_{00} = E^{-1}[N]$ and $ \mathcal{I} \cup c_{00} = E^{-1}[F]$
where $N$, $F$, $E$ are defined as before. Hence $ \mathcal{C} \cup c_{00} $ is $G_\delta$-set and $ \mathcal{I} \cup c_{00} $ is $\mathcal{F}_\sigma$-set. Thus $\mathcal{C}$ is $G_\delta$-set (because $c_{00}$ is $\mathcal{F}_\sigma$-set) and $\mathcal{I} \cup \mathcal{MC}$ is $\mathcal{F}_\sigma$. Moreover,  $\mathcal{I} = (\mathcal{I} \cup c_{00}) \cap (\mathcal{I} \cup \mathcal{MC}$) is $\mathcal{F}_\sigma$-set, too. By the density of $\mathcal{C}$, $\mathcal{C}$ is residual.  Since $\mathcal{I}$ is dense of the first category, it cannot be $\mathcal{G_\delta}$-set. For the same reason, $\mathcal{MC}$ also cannot be $\mathcal{G_\delta}$-set. Since $\mathcal{MC}$ is a difference of two $\mathcal{F_\sigma}$-sets, it is in the class $\mathcal{F}_{\sigma\delta}\cap\mathcal{G}_{\delta\sigma}$.
\end{proof}

\begin{remark}\emph{
In \cite{BG} it was shown the following similar result by the use of quite different methods: the set of bounded sequences, with the set of limit points homeomorphic to the Cantor set, is strongly $\mathfrak{c}$-algebrable and residual in $l^\infty$.}
\end{remark}

\section{Spaceability}

In this section we will show that $\mathcal I$ is spaceable while $\mathcal C$ is not spaceable. This shows that there is a subset $M$ of $\ell_1$ containing a dense $\mathcal G_\delta$ subset and such that it contains a linear subspace of dimension $\mathfrak{c}$, but $Y\setminus M\neq\emptyset$ for any infinitely dimensional closed subspace $Y$ of $\ell_1$.

\begin{theorem}\label{spaceability_of_I}
Let $\mathcal I_1$ be a subset of $\mathcal I$ which consists of those $x\in \ell_1$ for which $E(x)$ is an interval. Then $\mathcal I_1$ is spaceable.
\end{theorem}

\begin{proof}
Let $A_1,A_2,\dots$ be a partition of $\N$ into infinitely many infinite subsets. Let $A_n=\{k_n^1<k_n^2<k_n^3<\dots\}$. Define $x_n\in \ell_1$ in the following way. Let $x_n(k_n^j)=2^{-j}$ and $x_n(i)=0$ if $i\notin A_n$. Then $\Vert x_n\Vert_1=1$ and $\{x_n:x\in\N\}$ forms a normalised basic sequence. Let $Y$ be a closed linear space generated by $\{x_n:x\in\N\}$. Then
$$
y\in Y\iff\exists t\in \ell_1\Big(y=\sum_{n=1}^\infty t(n)x_n\Big).
$$
Since $E(x_n)=[0,1]$, then $E(\sum_{n=1}^\infty t(n)x_n)=\bigcup_{n=1}^\infty I_n$ where $I_n$ is an interval with endpoints $0$ and $t(n)$. Put $t^+(n)=\max\{t(n),0\}$ and $t^-(n)=\min\{-t(n),0\}$. Then $E(\sum_{n=1}^\infty t(n)x_n)=\left[\sum_{n=1}^\infty t^-(n),\sum_{n=1}^\infty t^+(n)\right]$ and the result follows.
\end{proof}

Let us remark the very recent result by Bernal-Gonz\'alez and Ord\'onez Cabrera \cite[Theorem 2.2]{BGOC}. The authors gave sufficient conditions for spaceability of sets in Banach spaces. Using that result, one can prove spaceability of $\mathcal I$ but it cannot be used to prove Theorem \ref{spaceability_of_I}, since the assumptions are not fulfilled.

However we do not know more results giving the sufficient conditions for a set in Banach space to not be spaceable. An interesting example of a non-spaceable set was given in the classical paper \cite{Gur66} by Gurarii where it was proved that the set of all differentiable functions from $C[0,1]$ is not spaceable. It is well known that the set of all differentiable functions in $C[0,1]$ is dense but meager. We will prove that even dense $\mathcal G_\delta$-sets in Banach spaces may not be spaceable.

\begin{theorem}
Let $Y$ be an infinitely dimensional closed subspace of $\ell_1$. Then there is $y\in Y$ such that $E(y)$ contains an interval.
\end{theorem}

\begin{proof}
Let $Y$ be an infinitely dimensional closed subspace of $\ell_1$. Let $\varepsilon_n\searrow 0$. Let $x_1$ be any nonzero element of $Y$ with $\Vert x_1\Vert_1=1+\varepsilon_1.$ Since $x_1\in \ell_1$, there is $n_1$ with $\sum_{n=n_1+1}^\infty |x_1(n)|\leq\varepsilon_1$. Let $E_1$ consist of finite sums $\sum_{n=1}^{n_1}\delta_nx_1(n)$ where $\delta_i\in\{0,1\}$. Then $E_1$ is a finite set with $\min E_1=\sum_{n=1}^{n_1}x_1^-(n)$, $\max E_1=\sum_{n=1}^{n_1}x_1^+(n)$ and $1\leq \max E_1-\min E_1\leq 1+\varepsilon_1$.

Let $Y_1=Y\cap\{x\in \ell_1:x(n)=0$ for every $n\leq n_1\}$. Since $\{x\in \ell_1:x(n)=0$ for every $n\leq n_1\}$ has a finite co-dimension, then $Y_1$ is infinitely dimensional. Let $x_2$ be any nonzero element of $Y_1$ with $\Vert x_2\Vert_1=1+\varepsilon_2.$ Since $x_2\in \ell_1$, there is $n_2>n_1$ with $\sum_{n=n_2+1}^\infty |x_i(n)|\leq\varepsilon_2$, $i=1,2$. Let $E_2$ consist of finite sums $\sum_{n=n_1+1}^{n_2}\delta_nx_2(n)$, where $\delta_i\in\{0,1\}$. Then $E_2$ is a finite set with $\min E_2=\sum_{n=n_1+1}^{n_2}x_2^-(n)$, $\max E_2=\sum_{n=n_1+1}^{n_2}x_2^+(n)$ and $1\leq \max E_2-\min E_2\leq 1+\varepsilon_2$.

Proceeding inductively, we define natural numbers $n_1<n_2<n_3<\dots$, infinitely dimensional closed spaces $Y\supset Y_1\supset Y_2\supset\dots$ such that $Y_k=\{x\in Y:x(n)=0$ for every $n\leq n_k\}$, nonzero elements $x_k\in Y_{k-1}$ with $\Vert x_k\Vert_1=1+\varepsilon_k$ and $\sum_{n=n_{k}+1}^\infty |x_i(n)|\leq\varepsilon_k$, $i=1,2,\dots,k$, and finite sets $E_k$ consisting of sums $\sum_{n=n_{k-1}+1}^{n_k}\delta_nx_k(n)$ where $\delta_i\in\{0,1\}$. Note that  $1\leq\on{diam}(E_k)\leq 1+\varepsilon_k$. Consider $y=\sum_{k=1}^\infty x_k/2^k$. We claim that $E(y)$ contains an interval $I:=[\min E_1,\max E_1]$.

Note that for any $t\in I$ there is $t_1\in E_1$ with $\vert t-t_1\vert\leq(1+\varepsilon_1)/2$. Since $1\leq\on{diam}(E_2)\leq  1+\varepsilon_2$, there is $t_2\in E_1+\frac12 E_2$ with $\vert t-t_2\vert\leq(1+\varepsilon_2)/2^2$. Hence, there is $\tilde{t}\in E(x_1+x_2/2)$ with $\vert t-\tilde{t}\vert\leq(1+\varepsilon_2)/2^2+\varepsilon_1$.
Since $1\leq\on{diam}(E_k)\leq  1+\varepsilon_k$, then inductively we can find  $t_k\in E_1+\frac12 E_2+\dots+\frac{1}{2^{k-1}}E_k$ with $\vert t-t_k\vert\leq(1+\varepsilon_k)/2^{k}$. Hence, there is $\tilde{t}\in E(x_1+x_2/2+\dots+x_k/2^{k-1})$ with $\vert t-\tilde{t}\vert\leq(1+\varepsilon_k)/2^k+\varepsilon_{k-1}+
\varepsilon_{k-1}/2+\dots+\varepsilon_{k-1}/2^{k-1}\leq(1+\varepsilon_k)/2^k+
2\varepsilon_{k-1}$. Since $E(y)$ is closed and it contains $E(x_1+x_2/2+\dots+x_k/2^{k-1})$, then $t\in E(y)$ and consequently $I\subset E(y)$.
\end{proof}

Immediately we get the following.
\begin{corollary}
The set $\mathcal C$ is not spaceable.
\end{corollary}

We end the paper with the list of open questions on the set $\mathcal{MC}$.
\begin{problem}
\begin{itemize}
\item[(i)] Is $\mathcal{MC}$ $\mathfrak{c}$-algebrable?
\item[(ii)] Is $\mathcal{MC}$ an $\mathcal F_\sigma$ subset of $\ell_1$?
\item[(iii)] Is $\mathcal{MC}$ spaceable?

\end{itemize}
\end{problem}

{\bf Acknowledgment.}
The second and the third authors have been supported by the Polish Ministry of Science and Higher Education Grant No.  N N201 414939 (2010-2013). We want to thank F. Prus-Wi\'sniowski who has informed us about the trichotomy of Guthrie and Nymann, and other references on subsums of series.

\end{document}